\documentclass[12pt]{article}
\usepackage[amsmath]{e-jc}

\usepackage{graphicx}

\dateline{March 11, 2026}{TBD}{TBD}


\Copyright{The authors. Released under the CC BY-ND license (International 4.0).}

\title{Broadcasting Agents and Adversary: \\A new variation on Cops and Robbers}

\author{William K. Moses Jr.\authornote{1}
\and
Amanda Redlich\authornote{2}
\and
Frederick Stock\authornote{3}
}

\authortext{1}{Department of Computer Science, Durham University, Durham, UK
   (\email{wkmjr3@gmail.com}).}
\authortext{2}{Department of Mathematics \& Statistics, UMass Lowell, Lowell, Massachusetts, USA (\email{amanda\_redlich@uml.edu}).}
\authortext{3}{Department of Computer Science, UMass Lowell, Lowell, Massachusetts, USA (\email{frederick\_stock@student.uml.edu}).}

\begin{document}

\maketitle

\begin{abstract}
 We introduce a new game played on graphs, ``Agents and Adversary".  This game is reminiscent of ``Cops and Robbers" but has some  fundamental differences.  We classify infinite families of graphs as Agents-win and Adversary-win.  We then define a new type of graph symmetry and use it to define a winning strategy for Adversary.  Finally, we give tight upper and lower bounds for Agents' time-to-win on several infinite families of graphs.
\end{abstract}

\section{Introduction}
Here we discuss a new graph game, played by a team of agents called Agents and an opponent called Adversary.  One member of Agents knows information at the beginning of the game.  Agents' goal is to share this information with the rest of the agents.  Adversary's goal is to prevent this by removing edges of the graph.  

This game was first described in Das, Giachoudis, Luccio, and Markou~\cite{DGLM20} (journal version~\cite{DGLM26}), although in different language. Building on previous work in~\cite{MRS25} and its extended version~\cite{MRSarxiv}, in this paper we classify graphs according to which player wins.  The main results are Theorems \ref{cutvx} and \ref{symmetry}, which give general constructions for Adversary win and Agents win graphs, and Theorem \ref{time}, which gives a tight bound on the length of the game play on certain graphs.

This information-sharing problem has a long history of being studied in computer science, starting with work by Slater, Cockayne, and Hedetniemi~\cite{SCH81}. A related problem where all agents have messages they want to transfer to each other, sometimes called gossip or information dissemination, has an even longer history of being studied in computer science. Hedetniemi, Hedetniemi, and Liestman~\cite{HHL88} is a nice survey on the early beginnings of these studies. A setting closer to the current one is the distributed message passing setting with dynamic networks and this information-sharing problem has been studied in this setting by Awerbuch and Even~\cite{AE84}, Casteigts, Flocchini, Mans, and Santoro~\cite{CFMS15}, Clementi, Pasquale, Monti, and Silvestri~\cite{CPMS09}, Kuhn, Lynch, and Oshman~\cite{KLO10}, and  O'Dell and Wattenhofer~\cite{OW05}. Another related setting where this problem has been studied is that of mobile agents on a static network attempting to broadcast information, as studied by Anaya, Chalopin, Czyzowicz, Labourel, Pelc, and Vax{\`e}s~\cite{ACCLPV16}, Czyzowicz, Diks, Moussi, and Rytter~\cite{CDMR18}, and Czyzowicz, Diks, Moussi, and Rytter~\cite{CDMR19}. 

Furthermore, this particular type of information-sharing question has obvious connections to cops-and-robbers.  In this case the cops (knowledgeable agents) and robbers (ignorant agents) are  cooperating and want to be co-located.  Cops-and-robbers, of course, also has a long history.  See Bonato and Nowakowski~\cite{CRbook} for an overview.  Early cops-and-robbers questions focused on classifying graphs as cop or robber wins (e.g.  Nowakowski and Winkler~\cite{FirstCR}, Frankl \cite{FRANKL1987301}).  Similarly, we begin our study by classifying graphs.  

In particular, we classify graphs according to how many agents are required for an Agents win, or if Adversary wins in all cases.  It is known that on any tree only two agents are enough, while on a cycle three are required (\cite{DGLM20, DGLM26}.  It seems natural to assume the minimum number of agents required is is related to the graph's edge-connectivity.  In fact a relationship between winning strategies and number of ``superfluous" edges was first conjectured in \cite{DGLM20}.  However, \cite{MRS25, MRSarxiv} show this is not the case.  Here we give some initial results on what graph properties \emph{do} determine winning strategies.

We start by formally defining the game and giving some illustrative examples in Section 2.  In Section 3, we give infinite classes of Agents win and Adversary win graphs, generated by specific graph constructions that preserve strategies.  We then describe a general winning strategy for Adversary, and a class of graphs to which it applies, in Section 4.  Finally, Section 5 contains conjectures and open questions.

\section{Definitions}
Here we are concerned with a game, $\text{BROADCAST}(G,n)$, played on a graph $G$ by two entities, ``Agents" and ``Adversary".  An Agents win happens when all agents become ``knowledgeable".  Otherwise, if Adversary can prevent all agents becoming knowledgeable, it is an Adversary win.  

This game was first introduced in \cite{DGLM20}, and discussed there in the context of distributed computing applications.  Informally, we can imagine the adversary as a hacker trying to prevent a team of agents from fully communicating with each other over a network.  A formal definition of the rules of the game follows.  

\begin{definition}
    The game $\text{BROADCAST}(G,k)$:
    \begin{itemize}
        \item SETUP: Adversary places $k-1$ ignorant agents and $1$ knowledgeable agent on $k$ distinct vertices of $G$.
        \item ADVERSARY MOVE AT TIME $t$: Adversary selects an arbitrary connected spanning subgraph of $G$, $G_t$.
        \item AGENTS MOVE AT TIME $t$: Each agent moves to a vertex in $G_t$ adjacent to its location at time $t-1$ or remains in its previous location.  If multiple agents are located at the same vertex at time $t$ and one is knowledgeable at time $t-1$, they all become knowledgeable at time $t$.
        \item GAME WINNER: If all agents become knowledgeable, the game ends with Agents win. If Adversary has a strategy to prevent all agents becoming knowledgeable, then it is an Adversary win. 
    \end{itemize}
\end{definition}

This game is similar to cops and robbers, with agents traveling through the graph and winning when co-located.  However, in this case the agents are all trying to be co-located, not avoid each other.  In a sense, the robbers are cooperating with the cops, and the graph itself is the ``opponent".

\subsection{Illustrative examples - Generalized theta graphs, grid}
Certain simple classes of graphs are easily categorized.  For example, on a tree $T$, Adversary is restricted to choosing the whole graph at every time step (due to the connectivity requirement).  Then a simple strategy such as ``all agents move to vertex $v$" leads to an Agents win.  Therefore $\text{BROADCAST}(T,k)$ is Agents win for all trees $T$ and all $k$.

On the other hand, Adversary has a simple winning strategy for $\text{BROADCAST}(C_m, 2)$ for all $m \geq 4$: Start with the knowledgeable agent at vertex $1$ and the ignorant agent at $\lfloor m/2 \rfloor$.  Choose the subgraph at time $t$ by removing the edge adjacent to the knowledgeable agent on the shortest path between the knowledgeable and ignorant agents (tiebreak at random).  This prevents the two agents from meeting, as if they were ever co-located, the timestep previous they would have been one edge apart, and Adversary would have removed that edge.

Further examples of strategies for grids and generalized theta graphs can be found in \cite{DGLM20}, \cite{DGLM26}, \cite{MRS25}, \cite{MRSarxiv}.  For completeness we restate some key theorems here.

\begin{theorem}[\cite{DGLM20}]
    \label{the:tree}
    For any tree $T$ and any $k$, $\text{BROADCAST}(T,k)$ is an Agents win.
\end{theorem}

\begin{theorem}[\cite{DGLM20}]
\label{the:cycle}
For any cycle $C_m$, $m>4$, $\text{BROADCAST}(C_m, 2)$ is an Adversary win and $\text{BROADCAST}(C_m, k)$ $k>2$ is an Agents win.
\end{theorem}
\begin{theorem}[\cite{DGLM20}]
\label{the:clique}
For any clique $K_m$,  $m>4$, $\text{BROADCAST}(K_m, k)$ is an Adversary win for $k<m-1$ and an Agents win for $k\geq m-1$.
\end{theorem}

\begin{theorem}[\cite{MRS25}]
\label{the:alg-works}
    For a generalized theta graph $G = \theta_{d_1, d_2,...,d_{\ell}}$ with $\ell \geq 1$ paths, each of length $d_i\geq 1$, $\text{BROADCAST}(G, k)$ is an Adversary win for $k \leq \ell$ and Agents win for $k>\ell$.
\end{theorem}

From these few results it seems clear that who wins $\text{BROADCAST}(G,n)$ depends on $n$, and that it is not obviously related to the number of edges or edge connectivity of the graph.  See \cite{MRS25, MRSarxiv} for further discussion of BROADCAST with respect to edge density and connectivity; here we take a different approach and look at other structural graph concepts.
\section{Strategy-preserving graph constructions}
A natural question is, which graph operations preserve Agents or Adversary wins?  For example, are there certain subgraphs or products that determine Agents or Adversary wins?  The answer is yes, there are.  Here we give a list of useful observations and strategies for graph constructions.
\subsection{Adversary-win graph constructions}
The first few observations are about Adversary's ability to leverage strategies on subgraphs.  For example, when there is a winning strategy on a connected spanning subgraph, Adversary can win by first reducing the playing field to that subgraph.  Note:  This lemma is tacitly used in the multidimensional grid cases of \cite{DGLM20} but is not stated or proved there; it is equivalent to Observation 3.3 in \cite{DGLM26}.

\begin{lemma}
    If $\text{BROADCAST}(G,k)$ is an Adversary win, then $\text{BROADCAST(H,k)}$ is also Adversary win for any $H$ such that $V(G)=V(H)$ and $E(G)\subseteq E(H)$.
\end{lemma}
\begin{proof}: If an Adversary strategy exists on $G$ for $k$ agents, the Adversary strategy for $H$ is to remove all edges in $E(H)\setminus E(G)$ and then follow the strategy for $G$.
\end{proof}
By using the Adversary-win Theorems \ref{the:cycle} and \ref{the:alg-works}, we can give explicit definitions of infinite families of Adversary-win graphs.

\begin{corollary}[Theorem~\ref{the:cycle}]
    $\text{BROADCAST}(G,k)$ is an Adversary win for any Hamiltonian graph $G$ on $\geq 5$ vertices and $k\leq 2$.
\end{corollary}
\begin{corollary}[Theorem~\ref{the:alg-works}]
   $\text{BROADCAST}(G,k)$ is an Adversary win for any graph with a spanning generalized theta subgraph with $\ell\geq k$ paths.
\end{corollary}
We now move on to more sophisticated constructions. Because Adversary chooses the initial placement of agents, certain types of constructions can ``generalize" a winning Adversary strategy to a larger class of graphs.  The first theorem of this type gives an infinitely family of Adversary-win graphs for any number of Agents.
\begin{theorem}
    For any $k$, if there exists some graph $G$ such that $\text{BROADCAST}(G,k)$ is an Adversary win, then there are infinitely many graphs $H$ for which $\text{BROADCAST}(H,k)$ is an Adversary win.
\end{theorem}

\begin{proof}
    Let $H'$ be any arbitrary graph.  Let $v$ be the vertex in $G$ on which Adversary would have placed the knowledgeable agent.  Let $H$ be some graph generated by identifying a vertex in $H'$ with the vertex $v \in G$; that is, the sum of $H'$ and $G$ at $v$.
    
    Adversary's winning strategy on $H$ is a modification of its winning strategy on $G$.  The initial placement of the $k$ agents in $H$ is the same as in $G$.  Adversary follows its original strategy for $G$ when agents are located in $G \subseteq H$.  If any agents are in $H-G$, Adversary acts as if they are located on $v$.  
    
    In particular, then, Adversary prevents the last ignorant agent from reaching $v$ when a knowledgeable agent is located on $v$ or in $H-G$, and vice-versa.  Therefore the last ignorant and any knowledgeable agent will not coincide on $G\subseteq H$ (guaranteed by Adversary's original strategy), nor will they coincide in $H-G$ (guaranteed by Adversary's prevention of a second agent passing through $v$).  Observing that there are infinitely many potential $H$ arising from this construction completes the proof.
\end{proof}

The Adversary in the previous theorem utilized the fact that $v$ is a cut-vertex between two parts of the graph to develop a winning strategy on the larger graph.  We can generalize this strategy to subgraphs as well.

\begin{theorem}\label{cutvx}
    Let $H$ be a graph with a unique cut-vertex $v$.  Let $G_1, \ldots G_s$ be the connected components of $H-v$.  Then for any $k$, if there exists $i$ and $k$ such that $\text{BROADCAST}(G_i\cup \{v\},k)$ is an Adversary win, $\text{BROADCAST}(H,k)$ is also an Adversary win.
\end{theorem}

\begin{proof}
    Adversary's strategy on $H$ is simply to follow its strategy on $G_i \cup \{v\}$, and treat all agents in $H-G_i-v$ as if they were on $v$.  Because Adversary has a winning strategy on $G_i \cup \{v\}$, at no time will a knowledgeable and the last ignorant agent be able to meet in $H-G_i-v$ (as then they would have met at $v$ in $G_i \cup \{v\}$).  Similarly, at no time will a knowledgeable and the last ignorant agent meet in $G_i \cup \{v\}$.  Therefore Adversary wins.
\end{proof}
For example, we now have the following:
\begin{corollary}[Theorem~\ref{the:clique}]
    BROADCAST(G, k) on any graph $G$ with a unique cut-vertex and $\omega(G)=r$ is an Adversary win for $k<r$.
\end{corollary}
\begin{corollary}[Theorem~\ref{the:cycle}]
    BROADCAST(G,k) on any graph $G$ with a unique cut-vertex and longest cycle length $r$ is an Adversary win for $k\leq 2$ and $r\geq 5$.
\end{corollary}
\begin{corollary}[Theorem~\ref{the:alg-works}]
     BROADCAST(G,k) on any graph $G$ with a unique cut-vertex and containing a generalized theta graph with $\ell$ paths is an Adversary win for $k\leq \ell$.
\end{corollary}
\subsection{Agents-win constructions}
Once cut vertices' impacts are understood, it is natural to consider cut edges.  While cut vertices were useful in generating Adversary strategies, cut edges are helpful to Agents.  Because Adversary's moves are restricted on cut edges, Agents can build a strategy around them.

\begin{theorem}\label{contract}
    Let $G$ be some connected graph and $H$ be a graph derived from $G$ by contracting along some or all of $G$'s cut-edges.  Then $\text{BROADCAST}(H,k)$ is an Agents win if and only if $\text{BROADCAST}(G,k)$ is an Agents win.
\end{theorem}
\begin{proof}
    Suppose $S$ is some winning strategy for Agents on $\text{BROADCAST}(G,k)$.  Define a new strategy, $S'$, for $\text{BROADCAST}(H,k)$ that follows the action of $S$ on $G$ \emph{except} for when an agent is directed by $S$ to travel along some contracted edge $e$ in $G$.  In that case, the strategy $S'$ directs the same agent to remain in one place.  Every vertex visited by an agent in $G$ is visited at the same time by the same agent in $H$.  Therefore at every time an agent in $G$ becomes knowledgeable, its counterpart in $H$ becomes knowledgeable as well.  Thus $S'$ is indeed a winning strategy for Agents.

    Supposed $H$ can be solved by $k$ agents following some strategy $S$.  Further suppose $H$ is the result of $c$ edge contractions in $G$. Define a new strategy $S'$ for $k$ agents on $G$ that at each time step $ci$ follows the strategy for $H$ at time $i$.  However, if an agent is directed to go to a neighbor in $H$ that is not a neighbor in $G$, i.e. a vertex that was connected to the end of a (series of) contracted edge(s), that agent uses time steps $ci, ci+1, ci+2 \ldots T\leq ci+c-1$ to reach it while the other agents remain still after time step $ci$.  At each step $ci, ci+1, ci+2 \ldots T\leq ci+c-1$ the agent is moving along a contracted edge; since only cut edges were contracted, they are guaranteed to be present in the Adversary's chosen subgraph.  Again, every vertex visited by an agent in $H$ at time $i$ is visited by an agent in $G$ at time $ci+c-1$.  Therefore, for every time $i$ an agent in $H$ becomes knowledgeable, its counterpart becomes knowledgeable by time $ci+c-1$.  Thus $S'$ is indeed a winning strategy for Agents.

\end{proof}

As before, this theorem gives a construction of an infinite family of Agents-win graphs:

\begin{theorem}
    For any $k$, if there exists some graph $G$ such that $\text{BROADCAST}(G,k)$ is an Agents win, then there are infinitely many graphs $H$ for which $\text{BROADCAST}(H,k)$ is an Agents win.
\end{theorem}

\begin{proof}
    Let $P_1, \ldots P_n$ be $n$ paths of arbitrary lengths and let $G$ be a graph on $n$ vertices such that $\text{BROADCAST}(G,k)$ is an Agents win.  Create a new graph $H$ by identifying an end vertex of each path $P_i$ to each vertex $v_i \in G$.  By contracting along the cut-edges of the $\{P_i\}$, $H$ becomes $G$.  Thus by the previous theorem, $\text{BROADCAST}(H,k)$ is an agent win.  By choosing infinitely many distinct sets $\{P_i\}$, infinitely many Agents win graphs are generated.
\end{proof}

This also gives some insight into which graphs are key to understanding $\text{BROADCAST}(G,k)$.  Since cut edges ``don't matter", it is sufficient to study $\text{BROADCAST}(G, k)$ where $G$ has edge-connectivity $\geq 2$.  This tacit observation is used in the analysis of cactus graphs in \cite{DGLM20} and \cite{DGLM26} as well as lollipop graphs in \cite{MRS25, MRSarxiv}.

\section{Adversary automorphism strategies}
The above strategies assumed ``black box" winning strategies for smaller graphs to generate strategies for infinite families.  We now develop explicit strategies by utilizing graph automorphisms.

We first give a strategy for Adversary on the grid graph that defeats an Agents strategy mentioned in \cite{DGLM20}. Of course, it does not defeat the Agents' winning strategy described in \cite{DGLM26}.  However, it is included here as an illustrative example of a type of Adversary strategy.  For further discussion of this type of Agents strategy and why it is important, see \cite{DGLM20}.

\begin{theorem}
    When Agents' strategy is ``choose an ignorant agent with the maximum numbers of ignorant agents on the path from this agent to the source during the current round of the algorithm, move all the agents on this path one step towards the source" then Adversary defeats Agents on $G$ the $m \times n$ grid when the number of agents is $\leq (m-1)(n-1)+2$ and $m>2$ or $n>2$.
\end{theorem}
\begin{proof}
    Without loss of generality, let $m\geq n$.  Index the vertices of the grid as $\{1, 2, \ldots m\} \times \{1, 2, \ldots n\}$.  The adversary's strategy is as follows: Place the knowledgeable agent at vertex $(1,1)$.  
    
    Place ignorant agents at $(2, j)$ for all $1\leq j \leq m$, $(2i+1, j)$ for all $1\leq i \leq n/2$ and $2 \leq j \leq m$, and $(2i, j)$ for all $2\leq i \leq n/2$ and $1\leq j \leq m-1$.  In other words, there are no ignorant agents in the first column, every vertex in the second column contains an ignorant agent, and every other column contains $m-1$ agents, missing either the top or bottom vertex alternately.

    Include the edges $\{(i,j), (i, j+1)\}$ for $1\leq i\leq n$, $1\leq j \leq m$; $\{(2i-1,m)(2i,m)\}$ for $1\leq i \leq n/2$; $\{(2i, 1), (2i+1, 1)\}$ for $1 \leq i \leq n/2$.  In other words, create a path winding up and down every column.

    The ignorant agents, following their strategy, will progress one edge towards the knowledgeable agent along this path.  The knowledgeable agent, following its strategy, will progress one edge towards the ignorant agents along this path.  The resulting positions are:
    \begin{itemize}
        \item Knowledgeable agent at $(1,2)$
        \item Ignorant agents at $(1,m)$, $(2i, j)$ for $2\leq j \leq m$, $(2i+1, j)$ for $1\leq j \leq m-1$
    \end{itemize}

    At this point Adversary ``flips" the graph to create a path in which distance decreases in the opposite direction.  To be precise, the new graph contains edges $\{(1,j), (1, j+1)\}$ for $j \neq 2$; $\{(i, j), (i, j+1)\}$ for $2 \leq i\leq n$; $\{(2i, m), (2i+1,m)\}$ for $1\leq i \leq n/2$; $\{(2i-1, 1),(2i,1)\}$ for $1 \leq i \leq n/2$; and $\{(1,1), (2,1)\}$.  In this configuration, the distance-decreasing strategy will lead the agents to return to their original positions.  At this point Adversary repeats its strategy, the agents move to their second position, etc.

\begin{figure}
    \centering
    \includegraphics[width=0.5\linewidth]{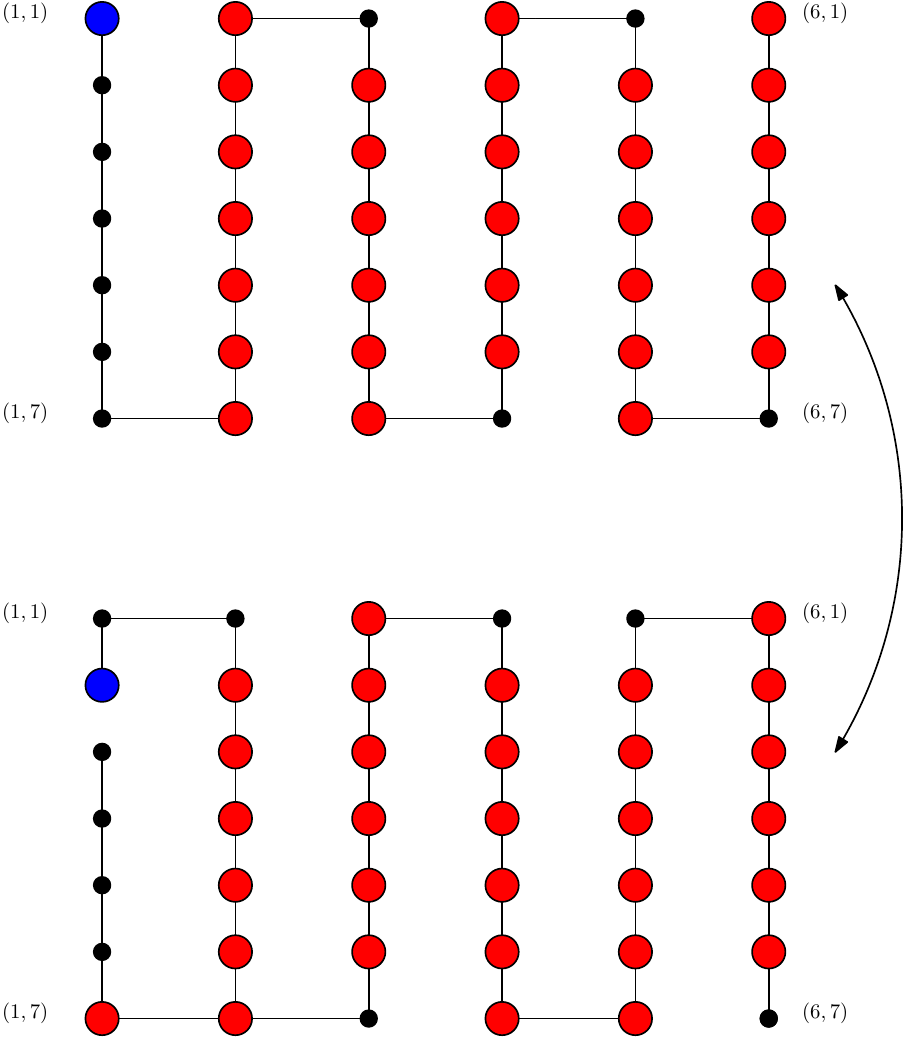}
    \caption{Alternating strategy on grid, knowledgeable agent is blue and ignorant red}
    \label{fig:placeholder}
\end{figure}

\end{proof}

The key idea in this proof was that Adversary could restrict Agents to a few repeating configurations.  We now generalize this idea to arbitrary graphs and Agents strategies.  First, we define potential Agents locations after a single time step.

\begin{definition}
    Say a set $S_2$ is a ``one-step set for $S$ with respect to $H$" if there is a function $t: S \to S_2$ with the property $v=t(v)$ or $\{v, t(v)\} \in E(H)$.  In other words, if agents are located at vertices in $S$ and the current graph is $H$, then the one-step sets are exactly the possible subsequent agent positions.
\end{definition}

We now define a graph property that allows Adversary to trap Agents in a repeating loop, as in the grid strategy. Note that this property is actually stronger than necessary (e.g. it does not hold for the grid).  It is not a vacuous property, though, as it holds for e.g. cliques.
\begin{definition} Say a graph has ``$k$-spanning tree symmetry" if for all $S\subseteq V(G)$ of size $k$, there exists a spanning tree $T_S$ such that for any one-step set ($S_2$) for $S$ with respect to $T_S$,  there is an extension of the one-step function $t$, $\varphi_{S, S_2}: T_S \to V(G)$, that is also an isomorphism of $T_S$.
\end{definition}
With these definitions in hand, it is straightforward to give Adversary's winning strategy.
\begin{theorem}\label{symmetry}
 For any graph $G$ with $k$-spanning tree symmetry, $\text{BROADCAST}(G,k)$ is an Adversary win.  This is true even if the rules are changed so that Agents, not Adversary, chooses the initial placement of agents.  
\end{theorem}
\begin{proof}
    At each stage, Adversary selects the spanning tree $T_S$ that demonstrates $k$-spanning tree symmetry based on the agents' current positions $S$.  The agents' next move is by definition a one-step set for $S$ with respect to $T_S$.  Adversary's next move is to select the spanning tree $\varphi_{S, S_2}(T_S)$.  As this new spanning tree is isomorphic to the original one, and the agents' positions on this new spanning tree are all isomorphic to their positions on the original spanning tree, no ignorant agent has decreased its distance to the knowledgeable agent.  Therefore this is a winning strategy for Adversary.
\end{proof}

\section{Strategies' time complexity}
As stated earlier, certain classes of graphs, like trees and cycles, are straightforward Agents wins.  A natural next question is how long does it take Agents to win?  We begin the analysis on a slightly more general version of $\text{BROADCAST}(G,k)$, in which there are an arbitrary number of knowledgeable and ignorant agents initially.  The proof is instructive when considering more complex graphs, as we do later.

\begin{theorem}
    For the path $P_n$ on $n$ vertices with $x$ ignorant agents and $y$ knowledgeable agents, Adversary can force a time of $(n-x-y)/2$ before an additional agent becomes knowledgeable but no longer.

\end{theorem}
\begin{proof}
    Adversary's strategy is to place the knowledgeable agents at one end of the path and the ignorant agents at the other.  At each time step the distance between any pair of agents can decrease by at most $2$.  Therefore one half the distance between them is a lower bound on the runtime.  This is achievable by both sets agents moving towards each other.  On the other hand, since Adversary may not remove any edges of a path, this strategy is always feasible.  The greatest distance between a set of vertices of size $x$ and $y$ on the path is $n-x-y$, which completes the proof.
\end{proof}

We now consider time until all (not just one additional) agents become knowledgeable.

\begin{theorem}
    On $P_n$ starting with $x$ ignorant agents and $y$ knowledgeable agents, Adversary can force a time of $(n-y)/2$ before Agents wins but no longer.  In particular, the time for Agents to win $\text{BROADCAST}(P_n, x+1)$ is tightly bounded by $(n-1)/2$
\end{theorem}

The proof of this theorem uses a lemma about Agents' time with respect to distances on a graph.

\begin{lemma}
    On any graph $G$ and for any number of agents $k$, the time for Agents to win $\text{BROADCAST}(G, k)$ is at least $r/2$, where $r$ is the longest shortest distance between a knowledgeable and ignorant agent at the beginning.
\end{lemma}
\begin{proof}
    Suppose not.  Then there exists a time step during which the distance between a knowledgeable and ignorant agent decreases by at least 3.  The distance between any two specific agents can never decrease by more than 2 in a single time step, so a previously-knowledgeable and ignorant agent can never have distance decrease by 3.  Therefore the only possibility is that an agent become knowledgeable that was previously ignorant, and that agent is at least 3 edges closer to an ignorant agent.  
    
    However a previously-ignorant agent $A$ only becomes knowledgeable when it is co-located with a knowledgeable agent $B$.  Therefore if $A$ is able to become 3 edges closer to an ignorant agent, by following the same strategy after co-location so could $B$.  As discussed above, it is impossible for $B$ to become 3 edges closer.
\end{proof}

With this lemma in hand, we return to the proof of the theorem.

\begin{proof}
Due to the ``one agent per vertex" restriction, the maximum distance between a single ignorant agent and the set of knowledgeable agents is at most $n-y$.  Again, the strategy of the ignorant agents moving towards the knowledgeable agent and vice-versa is best possible: the distance between ignorant and knowledgeable can never decrease by more than two per time step, and this strategy achieves that. 

If Adversary places the knowledgeable and ignorant agents at opposite ends of the path, the last agent can become knowledgeable in time $(n-y)/2$ by having them all move towards each other at each time step. 
\end{proof}

In fact, these observations can be generalized from paths to trees.

\begin{theorem}
    On a tree $T$ of diameter $d$, Adversary can force an Agents win of $\text{BROADCAST}(T, 2)$ to take at least time $d/2$.
\end{theorem}
\begin{proof}
    By the previous lemma, it is possible for Adversary to force a time of at least $d/2$, simply placing the two agents at opposite ends of the diameter.  Given that placement, the agents can meet in time $d/2$ by moving towards each other at each time step. 
\end{proof}

We now define a metric that will be useful in general time analyses:

\begin{definition}
    Call the $y$-set-diameter of a graph the maximum distance between a single vertex and a set of $y$ vertices.  
\end{definition}
For example, on $P_n$ the $y$-set-diameter is $n-y$.  On $K_n$ the $y$-set-diameter is $1$ for all $y$.  On a binary tree of depth $k$, the $y$-set-diameter is $2k$ for all $y \leq 2^{k-1}$. (Place the $y$-set on the leaves of one depth $k-1$ subtree and the single vertex on one of the other leaves.)

\begin{theorem}\label{time}
    For any graph $G$ with $y$-set-diameter $d$, and an initial configuration with $y$ knowledgeable and $x$ ignorant agents, Adversary has a strategy for forcing a time of $d/2$. In particular, Agents wins $\text{BROADCAST}(G, x+1)$ in time at least $d/2$. 
\end{theorem}
\begin{proof}
    Place the knowledgeable agents on the set of vertices achieving the $y$-set-diameter and at least one of the ignorant agents on the vertex achieving the $y$-set-diameter, and appeal to the previous lemma.
\end{proof}

\section{Open questions}
We have moved a step closer to categorizing when $\text{BROADCAST}(G,n)$ is an Agents or Adversary win.  It was already known that edge connectivity and density are irrelevant; we now know that $k$-spanning-tree symmetry and a graph's block structure are sufficient to determine winning strategies.  However, the question remains if there are specific graph metrics that are both necessary and sufficient to determine who wins.  For example, dismantlability is equivalent to cop-win.  Is there an analogous property in this setting?

Another question of interest is to see replace strategy with randomness.  For example, on $\text{BROADCAST}(K_m,n)$, if Adversary plays randomly, the question becomes equivalent to information sharing on a random connected graph on $m$ vertices.  This is a more realistic model of, e.g., storm interference with communication channels.  (The storm acts randomly rather than strategically.) It also allows for the application of decades of research on the Erd\H{o}s-R\'enyi random graph.

On the other hand, if Agents play randomly, the question is now about hitting time for random walks.  This could be a more realistic model of, e.g., communication in an unknown terrain.  (Each agent is unable to move strategically since they have no information about the overall graph.)  Again, the robust body of work on hitting time for random walks could be applied.

Finally, as discussed in Section 4, once a strategy is determined, it is of interest to know how long its implementation takes.  Among several different winning strategies, which is more efficient?  Among graphs of a particular size, which allows for the slowest or fastest resolution?  Again, this would be relevant to real-world strategizing.  It is also a natural question in algorithmic analysis.
\bibliographystyle{plain}
\bibliography{references.bib}
\end{document}